\documentclass[reqno,11pt]{amsart}
\usepackage[english]{babel}
\usepackage{amsmath,amsfonts,amsthm,amssymb,amsbsy,upref,color,graphicx,hyperref,enumerate,comment}
\usepackage[a4paper,margin=2.75truecm]{geometry}
\usepackage{stix}
\usepackage{tikz}
\usepackage{soul}
\usepackage{verbatim}
\DeclareMathOperator{\diam}{diam}
\DeclareMathOperator{\dive}{div}

\def\eps{{\varepsilon}}
\def\N{\mathbb{N}}
\def\O{\Omega}
\def\R{\mathbb{R}}

\def\HH{\mathcal{H}}

\def\la{\lambda}

\newcommand{\raggiocerchio}{2.7cm}
\newcommand{\distanzapunti}{6.1mm}
\newcommand{\MM}{9} 

\newcommand{\be}{\begin{equation}}
\newcommand{\ee}{\end{equation}}
\newcommand{\bib}[4]{\bibitem{#1}{\sc#2: }{\it#3. }{#4.}}
\def\res{\mathop{\hbox{\vrule height 7pt width .5pt depth 0pt \vrule height .5pt width 6pt depth 0pt}}\nolimits}

\newcommand{\wto}{\rightharpoonup}
\numberwithin{equation}{section}
\theoremstyle{plain}

\newtheorem{theo}{Theorem}[section]
\newtheorem{lemm}[theo]{Lemma}
\newtheorem{coro}[theo]{Corollary}
\newtheorem{prop}[theo]{Proposition}

\theoremstyle{remark}
\newtheorem{rema}[theo]{Remark}

\title[Inequalities between torsional rigidity and principal eigenvalue of the $p$-Laplacian]{Inequalities between torsional rigidity and principal eigenvalue of the $p$-Laplacian}

\author[L. Briani]{Luca Briani}

\author[G. Buttazzo]{Giuseppe Buttazzo}

\author[F. Prinari]{Francesca Prinari}

\date{}

\begin{document}

\begin{abstract}
We consider the torsional rigidity and the principal eigenvalue related to the $p$-Laplace operator. The goal is to find upper and lower bounds to products of suitable powers of the quantities above in various classes of domains. The limit cases $p=1$ and $p=\infty$ are also analyzed, which amount to consider the Cheeger constant of a domain and functionals involving the distance function from the boundary.
\end{abstract}

\maketitle



\section{Introduction}\label{sintro}

In this paper we consider the problem of minimizing or maximizing the quantity
$$\lambda_p^\alpha(\O)T_p^\beta(\O)$$
on the class of open sets $\O\subset\R^d$ having a prescribed Lebesgue measure, where $\alpha,\beta$ are two real parameters, and $\lambda_p(\O)$, $T_p(\O)$ are respectively the principal eigenvalue and the torsional rigidity, which are defined below, relative to the $p$-Laplace operator
$$\Delta_p u:=\dive\big(|\nabla u|^{p-2}\nabla u\big).$$

In all the paper, we use the following notation:
\begin{itemize}
\item $p'$ is the conjugate exponent of $p$ given by $p':=p/(p-1)$;
\item $\O\subset\R^d$ is an open set with finite Lebesgue measure $|\O|$;
\item $d_\O$ is the distance function from $\partial\O$
$$d_\O(x):=\inf\big\{|x-y|\ :\ y\in \partial\O \big\};$$
\item $\rho(\O)$ is the inradius of $\O$
$$\rho(\O):=\|d_\O\|_{L^\infty(\O)},$$
corresponding to the maximal radius of a ball contained in $\O$;
\item $\diam(\O)$ is the diameter of $\O$
$$\diam(\O):=\sup\big\{|x-y|\ :\ x,y\in\O\big\};$$
\item $P(\O)$ is the distributional perimeter of $\O$ in the De Giorgi sense, defined by
$$P(\O):=\sup\left\{\int_\O\dive\phi\,dx\ :\ \phi\in C^1_c(\R^d;\R^d),\ \|\phi\|_{L^\infty(\R^d)}\le1\right\};$$
\item $h(\O)$ is the Cheeger constant of $\O$, that we define in Section \ref{sp1};
\item $B_r$ is the open ball centered at the origin with radius $r$ in $\R^d$ and $\omega_d:=|B_1|$;
\item $\HH^{d-1}$ is the $d-1$ dimensional Hausdorff measure.
\end{itemize}

Given $1<p<\infty$, $T_p(\O)$ denotes the $p$-torsional rigidity of $\O$, defined by
\be\label{def.tor}
T_p(\O)=\max\bigg\{\Big[\int_\O|u|\,dx\Big]^{p}\Big[\int_\O|\nabla u|^p\,dx\Big]^{-1}\ :\ u\in W^{1,p}_0(\O),\ u\ne0\bigg\},
\ee
where $W^{1,p}_0(\O)$ stands for the usual Sobolev space obtained as the completion of the space $C^{\infty}_c(\O)$ with respect to the norm $\|\nabla u\|_{L^p(\O)}$. Equivalently, if $w_p$ is the unique weak solution of the nonlinear PDE
\be\label{eq.pdetor}
\begin{cases}
-\Delta_p w=1&\hbox{in }\O,\\
w\in W^{1,p}_0(\O),
\end{cases}
\ee
we can define $T_p(\O)$ as (see \cite{bra14}, Proposition 2.2):
\be\label{def.tor2}
T_p(\O)=\bigg(\int_\O w_p dx\bigg)^{p-1}.
\ee
Note that $w_p$ is a nonnegative function and \eqref{eq.pdetor} is the Euler-Lagrange equation of the variational problem
$$\min\Big\{J_p(u)\ :\ u\in W^{1,p}_0(\O)\Big\},$$
where
\be\label{def.Jp}
J_p(u):=\frac1p\int_\O|\nabla u|^p\,dx-\int_\O u\,dx.
\ee
Multiplicating by $w_p$ in \eqref{eq.pdetor} and integrating by parts gives
$$\int_\O w_p\,dx=\int_\O |\nabla w_p|^p\,dx=-p'J_p(w_p).$$
When $\O=B_1$, the solution $w_p$ to the boundary problem \eqref{eq.pdetor} is explicit and given by
\be\label{soltball}
w_{p}(x)=\frac{1-|x|^{p'}}{p'd^{1/(p-1)}}
\ee
which leads to
$$T_p(B_1)=\frac1d\Big(\frac{\omega_d}{p'+d}\Big)^{p-1}.$$

The $p$-principal eigenvalue $\lambda_p(\O)$ is defined through the Rayleigh quotient 
\be\label{def.lambda}
\lambda_p(\O)=\min\bigg\{\Big[\int_\O|\nabla u|^p\,dx\Big]\Big[\int_\O |u|^{p}\,dx\Big]^{-1}\ :\ u\in W^{1,p}_0(\O),\ u\ne0\bigg\}.
\ee
Equivalently, $\lambda_p(\O)$ denotes the least value $\lambda$ such that the nonlinear PDE
$$\begin{cases}
-\Delta_p u=\lambda |u|^{p-2}u&\hbox{in }\O,\\
u\in W^{1,p}_0(\O),
\end{cases}$$
has a nonzero solution; we recall that in dimension $1$ we have (see for instance \cite{Kaji})
\be\label{lball}
\lambda_p(-1,1)=\left(\frac{\pi_p}{2}\right)^p\qquad\text{where }\pi_p=2\pi\frac{(p-1)^{1/p}}{p\sin(\pi/p)},
\ee
while in higher dimension the following estimate holds true, see (\cite{H} Theorem 3.1):
\be\label{lballHD}
\la_p(B_1)\le \frac{(p+1)(p+2)\dots (p+d)}{d!}\; .
\ee

It is easy to see that the two quantities above scale as
\be\label{scaling}
\lambda_p(t\O)=t^{-p}\lambda_p(\O),\qquad T_p(t\O)=t^{p+d(p-1)}T_p(\O).
\ee
By using a symmetrization argument and the so-called P\'olya-Szeg\"o principle (see \cite{he06}) it is possible to prove that balls maximize (respectively minimize) $T_p$ (respectively $\lambda_p$) among all sets of prescribed Lebesgue measure, which can be written in a scaling free form as
\be\label{FK}
\lambda_p(B)|B|^{p/d}\le\lambda_p(\O)|\O|^{p/d},\qquad T_p(\O)|\O|^{1-p-p/d}\le T_p(B)|B|^{1-p-p/d},
\ee
where $B$ is any ball in $\R^d$. The inequalities \eqref{FK} are known respectively as \textit{Faber-Krahn inequality} and \textit{Saint-Venant inequality}. 

Moreover, we have:
\be\label{inftorall}
\inf\big\{T_p(\O)\ :\ \O\text{ open in }\R^d,\ |\O|=1\big\}=0,
\ee
\be\label{suplaall} 
\sup\big\{\lambda_p(\O)\ :\O\text{ open in }\R^d,\ |\O|=1\big\}=+\infty.
\ee
To prove \eqref{inftorall} and \eqref{suplaall} it is enough to take into account of the scaling properties \eqref{scaling} and use the fact that if $\O$ is the disjoint union of a family of open sets $\O_i$ with $i\in I$, then
\be\label{eq.toradd}
T_p^{1/(p-1)}(\O)=\sum_{i\in I}T_p^{1/(p-1)}(\O_i),\qquad\lambda_p(\O)=\inf_{i\in I}\lambda_p(\O_i).
\ee
Then, choosing $\O_n$ as the disjoint union of $n$ balls with measure $1/n$ each and taking the limit as $n\to\infty$, gives
$$T_p(\O_n)=\omega_d^{1-p-p/d}n^{-p/d}T_p(B_1)\to0$$
and
$$\lambda_p(\O_n)=\omega^{p/d}n^{p/d}\lambda_p(B_1)\to+\infty.$$

Thus, a characterization of $\inf/\sup$ of the quantity $\lambda_p^\alpha(\O)T_p^\beta(\O)$, among the domains $\O\subset\R^d$ with unitary measure, 
when $\alpha=0$ or $\beta=0$ or $\alpha\beta<0$, follows by \eqref{FK}, \eqref{inftorall} and \eqref{suplaall}. 

It remains to consider the case $\alpha>0$ and $\beta>0$. Setting $q=\beta/\alpha>0$ we can limit ourselves to deal with the quantity
$$\lambda_p(\O)T_p^q(\O)\,.$$

Using the scaling properties \eqref{scaling} we can remove the constraint of prescribed Lebesgue measure on $\O$ by normalizing the quantity $\lambda_p(\O)T_p^q(\O)$, multiplying it by a suitable power of $|\O|$. We then end up with the scaling invariant shape functional
$$F_{p,q}(\O)=\frac{\lambda_p(\O)T_p^q(\O)}{|\O|^{\alpha(p,q,d)}}\qquad\text{with }\alpha(p,q,d):=q(p-1)+\frac{p(q-1)}{d}\;,$$
that we want to minimize or maximize over the class of open sets $\O\subset\R^d$ with $0<|\O|<\infty$.

The limit cases, when $p=1$ and $p=+\infty$, are also meaningful. When $p\to 1$ the quantities $\lambda_p(\O)$ and $T_p(\O)$ are related to the notion of Cheeger constant $h(\O)$, see definition \eqref{def.cheeger}. In particular we obtain as a natural ``limit'' functional 
$$F_{1,q}(\O)=\left(h(\O)|\O|^{1/d}\right)^{1-q}$$
whose optimization problems are well studied in the literature. Concerning the case $p=+\infty$, we show that the family $F_{p,q}^{1/p}$ pointwise converges, as $p\to\infty$, to the shape functional
$$F_{\infty,q}(\O)=\frac{\big(\intbar_\O d_{\O}(x)\,dx\big)^q}{\rho(\O)|\O|^{(q-1)/d}},$$
and we study the related optimization problems in the classes of all domains $\O$ and in the one of convex domains.

The study of the functionals $F_{p,q}$ has been already considered in the literature. The case when $p=2$ has been extensively discussed in \cite{bbp20}, \cite{bbv15}, \cite{bfnt16} and \cite{bfnt19} (see also \cite{bra20}) and our results can be seen as natural extensions.

The paper is organized as follows. In the first three sections we study the optimization problems for $F_{p,q}$, when $1<p<\infty$ and in different classes of domains. More precisely: in Section \ref{sgene} we consider the class of all open sets of $\R^d$ with finite Lebesgue measure, in Section \ref{sconvex} we consider the class of bounded convex open sets and in Section \ref{sthin} that of thin domains which will be suitable defined. The analysis of the optimization problems in the extremal cases (respectively when $p=1$ and $p=+\infty$) are contained in Section \ref{sp1} and \ref{spinfty}. Finally Section \ref{further} contains a list of several open problems which we believe may be interest for future researches. For the sake of completeness we add an appendix section devoted to clarify the assumptions we use for the limit case of Section \ref{spinfty}.

\section{Optimization for general domains}\label{sgene}

The crucial inequality to provide a lower bound to $F_{p,q}$ is the Kohler-Jobin inequality, first proved for $p=2$ in \cite{kj78a,kj78b} and then for a general $p$ in \cite{bra14}, which asserts that balls minimize principal frequency  among all sets of prescribed torsional rigidity. More precisely we have
\be\label{eq.kjineq}
\lambda_p(B)T_p^{p'/(p'+d)}(B)\le\lambda_p(\O)T_p^{p'/(p'+d)}(\O).
\ee

\begin{prop}\label{prop.infall}
Let $1<p<+\infty$. Then
$$\begin{cases}
\min\big\{F_{p,q}(\O)\ :\ \O\text{ open in }\R^d,\ 0<|\O|<\infty\big\}=F_{p,q}(B)&\hbox{if }\ 0<q\le p'/(p'+d);\\
\inf\big\{F_{p,q}(\O)\ :\ \O\text{ open in }\R^d,\ 0<|\O|<\infty\big\}=0\,&\hbox{if }\ q>p'/(p'+d).
\end{cases}$$
where $B$ is any ball in $\R^d$.
\end{prop}

\begin{proof}
We denote for the sake of brevity $\bar q=p'/(p'+d)$.
Notice that 
$$\alpha(p,q,d)=\frac{[d(p-1)+p][q-\bar q]}{d}$$
and thus
$$F_{p,q}(\O)=\lambda(\O)T^{\bar q}(\O)\Big[\frac{T(\O)}{|\O|^{[d(p-1)+p]/d}}\Big]^{q-\bar q}\,,$$
By Kohler-Jobin inequality \eqref{eq.kjineq} and Saint-Venant inequality \eqref{FK} we get the thesis for $0<q\le \bar q$. Now, let $\O$ be the disjoint union of $B_1$ and $N$ disjoint balls of radius $\eps\in (0,1]$. Taking into account of \eqref{eq.toradd} we have
$$F_{p,q}(\O)=F_{p,q}(B_1)\,\frac{(1+N\eps^{d+p/(p-1)})^{q(p-1)}}{(1+N\eps^d)^{(d(p-1)+p)(q-\bar q)/d}}\,.$$
Taking now $N\eps^{d+p/(p-1)}=1$ gives
$$F_{p,q}(\O)\le F_{p,q}(B_1)\frac{2^{q(p-1)}}{(1+\eps^{-p/(p-1)})^{(d(p-1)+p)(q-\bar q)/d}}\,,$$
which vanishes as $\eps\to0$ as soon as $q>\bar q$.
\end{proof}

In dealing with the supremum of $F_{p,q}$ a natural threshold arises from the \textit{Polya inequality} whose brief proof we recall.

\begin{prop}\label{prop.polya}
For every $\O\subset\R^d$ with $0<|\O|<+\infty$ and every $1<p<+\infty$ we have
\be\label{eq.polya}
F_{p,1}(\O)=\frac{\lambda_p(\O)T_p(\O)}{|\O|^{p-1}}\le 1.
\ee
\end{prop}

\begin{proof}
Let $w_\O$ be the solution to \eqref{eq.pdetor}. By the definition of $\lambda_p(\O)$ and by H\"older inequality we have
$$\lambda_p(\O)\le \frac{\int_\O |\nabla w_p|^p dx}{\int_\O w_p^p dx}=\frac{\int_\O w_p dx}{\int_\O w_p^p dx}\le \frac{|\O|^{p-1}}{\left(\int_\O w_p dx \right)^{p-1}}.$$
The conclusion follows by \eqref{def.tor2}.
\end{proof}

\begin{prop}\label{prop.supall}
Let $1<p<\infty$. Then
$$\begin{cases}
\sup\big\{F_{p,q}(\O)\ :\ \O\text{ open in }\R^d,\ 0<|\O|<\infty\big\}=+\infty&\hbox{if }0<q<1;\\
\sup\big\{F_{p,q}(\O)\ :\ \O\text{ open in }\R^d,\ 0<|\O|<\infty\big\}\le T^{q-1}_p(B)/|B|^{(q-1)(p-1+p/d)}&\hbox{if }q\ge1.
\end{cases}$$
\end{prop}

\begin{proof}
Let $\O_N$ be the disjoint union of $N$ balls of unitary radius. By \eqref{eq.toradd} we have
$$F_{p,q}(\O_N)=N^{(1-q)(p-1+p/d)}F_{p,q}(B_1).$$
Taking the limit as $N\to\infty$ we have $F_{p,q}(\O_N)\to+\infty$ whenever $0<q<1$. Moreover, when $q\ge 1$, using Proposition \ref{prop.polya} and the Saint-Venant inequality \eqref{FK}, we have
$$F_{p,q}(\O)=F_{p,1}(\O)\Big(\frac{T_p(\O)}{|\O|^{p-1+p/d}}\Big)^{q-1}\le\Big(\frac{T_p(B)}{|B|^{p-1+p/d}}\Big)^{q-1}\,.$$
which concludes the proof.
\end{proof}

When $p=2$ and $q=1$ the upper bound given in the Proposition \ref{prop.supall} is sharp as first proved in \cite{bfnt16}. Using the theory of \textit{capacitary measures}, a shorther proof was given in \cite{bbp20}. The latter extends, naturally, to the case when $p\le d$ and $q=1$ as we show in the proposition below. 

\begin{prop}\label{sharppolya}
Let $1<p\le d$. Then 
$$\sup\big\{F_{p,1}(\O)\ :\ \O\subset\R^d\hbox{ open, }0<|\O|<+\infty\big\}=1.$$
\end{prop}

\begin{proof}
By repeating the construction made in \cite{dmma95} (see also Remark 4.3.11 and Example 4.3.12 of \cite{bubu05}, and references therein) we have that for every $p$\textit{-capacitary} measure $\mu$ (that is a nonnegative Borel measure, possibly taking the value $+\infty$, with $\mathrm{cap}_p(E)=0\Longrightarrow\mu(E)=0$) there exists a sequence $(\O_n)$ of (smooth) domains such that
$$\lambda_p(\O_n)\to\lambda_p(\mu),\qquad T_p(\O_n)\to T_p(\mu),\qquad|\O_n|\to|\{\mu<+\infty\}|,$$
where
\[\begin{split}
&\lambda_p(\mu)=\min\bigg\{\int|\nabla u|^p\,dx+\int u^p\,d\mu\ :\ u\in W^{1,p}(\R^d)\cap L^p_\mu,\ \int u^p\,dx=1\bigg\},\\
&T_p(\mu)=\max\bigg\{\bigg[\int u\,dx\bigg]^p\bigg[\int|\nabla u|^p\,dx+\int u^p\,d\mu\bigg]^{-1}\ :\ u\in W^{1,p}(\R^d)\cap L^p_\mu\setminus\{0\}\bigg\}.
\end{split}\]
Taking the ball $B_1$ and $\mu_c=c\,dx\res B_1$ for every $c>0$, we have
\[
\sup_\O\bigg[\frac{\lambda_p(\O)T_p(\O)}{|\O|^{p-1}}\bigg]=\sup_\mu\bigg[\frac{\lambda_p(\mu)T_p(\mu)}{|\{\mu<+\infty\}|^{p-1}}\bigg]\ge\sup_{c>0}\bigg[\frac{\lambda_p(\mu_c)T_p(\mu_c)}{|B_1|^{p-1}}\bigg]\,.
\]
Clearly $\lambda_p(\mu_c)=c+\lambda_p(B_1)$. Now, consider for $\delta>0$ the function
\[u_\delta(x)=
\begin{cases}
1&\hbox{if }|x|\le 1-\delta,\\
(1-|x|)/\delta&\hbox{if }|x|>1-\delta\,.
\end{cases}\]
We have
\[\begin{split}
T_p(\mu_c)
&\ge\Big[\int_{B_1}u_\delta\,dx\Big]^p\Big[\int_{B_1}|\nabla u_\delta|^p\,dx+c\int_{B_1}u_\delta^p\,dx\Big]^{-1}\\
&\ge\Big[\omega_d(1-\delta)^d\Big]^p\Big[\delta^{-p}\omega_d\big(1-(1-\delta)^d\big)+c\omega_d\Big]^{-1}\\
&\ge\omega_d^{p-1}(1-\delta)^{pd}\Big[\delta^{-p}+c\Big]^{-1}\,.
\end{split}\]
Therefore
$$\frac{\lambda_p(\mu_c)T_p(\mu_c)}{|B_1|^{p-1}}\ge\frac{(c+\lambda(B_1))(1-\delta)^{pd}}{\delta^{-p}+c}\,.$$
By letting $c\to+\infty$ and then $\delta\downarrow 0$ we obtain the thesis.
\end{proof}

\section{Optimization in Convex Domains}\label{sconvex}

We now deal with the optimization problems in the class of convex domains. Notice that adding in \eqref{inftorall} and in \eqref{suplaall} a convexity constraint on the admissible domains $\O$ does not change the values of $\inf$ and $\sup$. To see this one can take a unit measure normalization of the following convex domains ({\it slab shape})
\be\label{slab}
C_{A,\eps}:=A\times(-\eps,\eps)
\ee
being $A$ a convex $d-1$ dimensional open set with finite $d-1$ dimensional measure and use the following Lemma, which will be proved in a slightly more general version in Proposition \ref{stimethin} of Section \ref{sthin}.

\begin{lemm}\label{stime}
Let $A\subset\R^{d-1}$ be a bounded open set and let $\eps >0$. Let $C_{A,\eps}:=A\times(-\eps,\eps)$. Then we have
$$T_p(C_{A,\eps})\le \left(\HH^{d-1}(A)\right)^{p-1}\eps^{2p-1}\left(\frac{2}{p'+1}\right)^{p-1},\qquad\lambda_p(C_{A,\eps})\ge \eps^{-p}\left(\frac{\pi_p}{2}\right)^p,$$
where $\pi_p$ is given in \eqref{lball}. In addition, as $\eps\to0$, we have
$$T_p(C_{A,\eps})\approx\left(\HH^{d-1}(A)\right)^{p-1}\eps^{2p-1}\left(\frac{2}{p'+1}\right)^{p-1}, \qquad\la_p(C_{A,\eps})\approx \eps^{-p}\left(\frac{\pi_p}{2}\right)^p.$$
\end{lemm}

By using the previous lemma we have also
$$\lim_{\eps\to0}F_{p,q}(C_{A,\eps})=
\begin{cases}
0&\hbox{ if }q>1,\\
+\infty&\hbox{ if }q<1.
\end{cases}$$
Hence the only interesting optimization problems in the class of convex domains are the following ones 
$$\inf\{F_{p,q}(\O)\ :\ \O\subset\R^d\hbox{ open, convex, bounded}\},\qquad\hbox{with }q\le1,$$
$$\sup\{F_{p,q}(\O)\ :\ \O\subset\R^d\hbox{ open, convex, bounded}\},\qquad\hbox{with }q\ge1.$$
We denote respectively by $m_{p,q}$ and $M_{p,q}$ the two quantities above. 

With the convexity constraint, the so called Hersch-Protter inequality holds (for a proof see for instance \cite{bra18} and \cite{Kaji}):
\be\label{HP1}
\lambda_p(\O)\ge\left(\frac{\pi_p}{2\rho(\O)}\right)^p.
\ee
Moreover, the $p$-torsional rigidity of a bounded convex open set satisfies the following generalization of Makai inequality (see \cite{DPGGLB}, Theorem 4.3):
\be\label{Tors-Inr}
\frac{T_p(\O)}{|\O|^{p-1}}\le\frac{\rho^p(\O)}{(p'+1)^{p-1}}.
\ee
Both inequalities are sharp and the equality is asymptotically attained by taking, for instance, the sequence $C_{A,\eps}$ of Lemma \ref{stime}. Taking advantage of \eqref{HP1} and \eqref{Tors-Inr} we can show the following bounds.

\begin{prop}\label{prop.p1convex}
Let $1<p<+\infty$. Then
\begin{align}
&m_{p,1}\ge\frac1d\Big(\frac{\pi_p}{2}\Big)^p\Big(\frac{1}{d+p'}\Big)^{p-1},\label{infp1conv}\\
&M_{p,1}\le\min\bigg\{1,\Big(\frac{1}{p'+1}\Big)^{p-1} \la_p(B_1)\bigg\}.\label{sup.p1conv}
\end{align}
\end{prop}

\begin{proof}
Let $\O\subset\R^{d}$ be any bounded convex set. Without loss of generality, we can suppose $0\in\O$. We denote by $j_{\O}(x)$ the \textit{Minkowski functional} (also known as \textit{gauge function}) of $\O$, that is
$$
j_{\O}(x):=\inf\left\{r>0: x\in r\O\right\}.
$$
The main properties of $j_{\O}$ are summarized in Lemma 2.3 of \cite{bra20}. In particular we recall that $j_\O$ is a convex, Lipschitz, $1$-positively homogeneous function, $\HH^{d-1}$-a.e. differentiable in $\partial\O$, and satisfies
\be\label{min1}
|\nabla j_{\O}(x)|^{-1}=x\cdot\nu_{\O}(x), \qquad \text{for }\HH^{d-1}\text{-a.e. }x\in\partial\O,
\ee
being $\nu_{\O}(x)$ the outer normal unit versor at the point $x\in\partial\O$.
We consider
$$u(x):=1-j^{p'}_\O(x)\in W^{1,p}_0(\O).$$
By using coarea formula \eqref{min1} and the divergence theorem it is easy to prove that
$$\int_\O u(x) dx=|\O|-\int_0^{1}t^{p'+d-1}dt\int_{\partial\O}|j_{\O}(x)|^{-1}d\HH^{d-1}(x)=\frac{p'}{d+p'}|\O|,$$
and
$$\int_\O|\nabla u(x)|^{p}dx=\frac{p'}{p+d}\int_{\partial\O}|\nabla j_{\O}(x)|^{p-1}d\HH^{d-1}(x)\ge \frac{dp'^{p}}{p'+d}|\O|\rho^{p}(\O),$$
where the last inequality follows by the fact that
$$\rho(\O)\le x\cdot \nu_{\O}(x), \text{for }\HH^{d-1}\text{-a.e. }x\in\partial\O,$$
see Lemma 2.1 in \cite{bra20}. Hence by testing \eqref{def.tor} with the function $u$ we have
\be\label{Tpbasso}
\frac{T_p(\O)}{\rho(\O)^p|\O|^{p-1}}\ge\frac1d\big(\frac{1}{ d+p'}\big)^{p-1}.
\ee
Taking into account \eqref{HP1}, we obtain
$$F_{p,1}(\O)\ge\frac1d\big(\frac{1}{ d+p'}\big)^{p-1}\la_p(\O)\rho(\O)^p\ge\frac1d \Big(\frac{\pi_p}{2}\Big)^p \big(\frac{1}{d+p'}\big)^{p-1},$$
which proves \eqref{infp1conv}.

To prove the second inequality we use \eqref{Tors-Inr} together with the inequality
$$\la_p(\O)\le \la_p(B_1)\rho(\O)^{-p},$$
to obtain
$$F_{p,1}(\O)\le\Big(\frac{1}{p'+1}\Big)^{p-1}\rho^p(\O)\la_p(\O)\le\Big(\frac{1}{p'+1}\Big)^{p-1}\la_p(B_1),$$
which, together with Proposition \ref{eq.polya}, gives \eqref{sup.p1conv}.
\end{proof}

\begin{rema}
We stress here that inequality \eqref{Tpbasso} has been already proved in \cite{bgm17} and \cite{DPGGLB}. However, their results are given in the more general anisotropic setting where the proofs become more involved.
\end{rema}

\begin{rema}
Combining inequalities \eqref{sup.p1conv} and \eqref{lballHD}, we obtain
$$F_{p,1}(\O)\le\Big(\frac{1}{p'+1}\Big)^{p-1}\frac{(p+1)(p+2)\cdots(p+d)}{d!}\le\frac{(1+p)^d}{2^{p-1}}.$$
Thereby, as soon as $p$ is large enough, we have $M_{p,1}<1$.
\end{rema}

When $q\ne1$ the values $m_{p,q}$ and $M_{p,q}$ are achieved by some optimal domains, as shown in the next theorem.

\begin{theo}
Let $1<p<+\infty$. Then 
$$\begin{cases}
m_{p,q}\ge m_{p,1}T^{q-1}_p(B)/|B|^{(d(p-1)+p)(q-1)/d}&\hbox{ if }q<1,\\
M_{p,q}\le M_{p,1}T^{q-1}_p(B)/|B|^{(d(p-1)+p)(q-1)/d}&\hbox{ if }q>1.
\end{cases}$$
Moreover, there exist convex domains $\O^m_{p,q}$ and $\O^M_{p,q}$ such that
$$\begin{cases}
F_{p,q}(\O^m_{p,q})=m_{p,q}&\hbox{ if }q<1,\\
F_{p,q}(\O^M_{p,q})=M_{p,q}&\hbox{ if }q>1.
\end{cases}$$
\end{theo}

\begin{proof}
The first part follows at once using Saint-Venant inequality \eqref{FK} togheter with the equality
$$F_{p,q}(\O)=F_{p,1}(\O)\Big(\frac{T_p(\O)}{|\O|^{p-1+p/d}}\Big)^{q-1}.$$
Concerning the existence of optimal convex domains, we can repeat the argument used in \cite{bbp20}. First we notice that
\be\label{ex1}
F_{p,q}(\O)=\frac{F^q_{p,1}(\O)\la_p(\O)^{1-q}}{|\O|^{p(q-1)/d}}.
\ee
Moreover, any convex open set $\O$ contains a two-sided cone with base area equal to a $d-1$ dimensional disk of radius $\rho(\O)$ and total height equal to $\diam(\O)$, hence
\be\label{cones}
|\O|\ge d^{-1}\omega_{d-1}\diam(\O)\rho(\O)^{d-1}.
\ee
Thus, suppose  $0<q<1$ and let $(\O_n)$ be a minimizing sequence for $F_{p,q}$ made up of convex domains. By scaling invariance we can suppose $\rho(\O_n)=1$. For $n$ large enough we have $F_{p,q}(\O_n)\le F_{p,q}(B)$. Using \eqref{HP1} and \eqref{ex1} we have
$$\frac{F_{p,q}(B)}{m^q_{p,1}}\ge \left(\frac{\pi_p}{2}\right)^{p(1-q)}|\O_n|^{p(1-q)/d}.$$
Combining the last estimate with \eqref{cones} we have
$$\sup_n\diam(\O_n)<+\infty.$$ 
Hence, up to translations, the whole sequence $(\O_n)$ is contained in a compact set and we can extract a subsequence $(\O_{n_k})$ which converges in both Hausdorff and co-Hausdorff distance to some $\O^m_{p,q}$ (see \cite{hepi05}, for details about these convergences). Using the well-known continuity properties for $\la_p$, $T_p$ and Lebesgue measure with respect to Hausdorff metrics on the class of bounded convex sets, we conclude that
$$m_{p,q}=\lim_{n\to\infty}F_{p,q}(\O_n)=F_{p,q}(\O^m_{p,q}).$$
If $q>1$ we can follow the similar strategy and consider a maximizing sequence $(\O_n)$ with unitary inradius. By \eqref{ex1} and \eqref{HP1} we have, for $n$ large enough,
$$F_{p,q}(B)\le F_{p,q}(\O_n)\le M^q_{p,1}\left(\frac{\pi_p}{2}\right)^{1-q}\left(\frac{1}{|\O_n|^{1/d}}\right)^{p(q-1)},$$
which, thanks to \eqref{cones}, implies again $\sup_n\diam(\O_n)<+\infty$.
\end{proof}

\section{Optimization for thin domains}\label{sthin}

In this section we study the optimization problems for the functionals $F_{p,1}$ in the class of the so-called thin domains, which has been already considered in \cite{bbp20} for $p=2$. By a thin domain we mean a family of open sets $(\O_\eps)_{\eps>0}$, of the form
\be\label{thin}
\O_\eps:=\big\{(x,y)\in A\times\R\ :\ \eps h_-(x)<y<\eps h_+(x)\big\}\; ,\ee
where $A$ is $(d-1)$-dimensional open set, $h_-,h_+$ are real bounded measurable functions defined on $A$ and $\eps$ is a small parameter. We assume $h_+\ge h_-$ and we denote by $h(x)$ the {\it local thickness function}
$$h(x)=h_+(x)-h_-(x)\ge0.$$
Moreover we say that the thin domain $(\O_\eps)_{\eps>0}$ is convex if the corresponding domain $A$ is convex and the local thickness function $h$ is concave. The volume of $\O_\eps$ is clearly given by
$$|\O_\eps|=\eps\int_A h(x)\,dx,$$
while we can compute the behaviour (as $\eps\to0$) of $T_p(\O_\eps)$ and $\la_p(\O_\eps)$ by means of the following proposition (in the case $p=2$ a more refined asymptotics can be found in \cite{bofr10}, \cite{bofr13}).
From now on, we write the norms $\|\cdot\|_{p}$, omitting the dependence on the domain. 

\begin{prop}\label{stimethin}
Let $A\subset\R^{d-1}$ be an open set with finite $\HH^{d-1}$-measure and $h_-,h_+\in C^{1}(A)$ with $h_+>h_-$. Let $\O_\eps$ be defined by \eqref{thin}. We have
\be\label{ineqthin}
T_p(\O_\eps)\le\frac{\eps^{2p-1}}{2^p}\left(\frac{1}{p'+1}\right)^{p-1}\left(\int_A h^{p'+1}dx\right)^{p-1},\qquad \la_p(\O_\eps)\ge\eps^{-p}\left(\frac{\pi_p}{\|h\|_\infty}\right)^p,
\ee
where $\pi_p$ is given in \eqref{lball}.
In addition, as $\eps\to0$, we have
\be\label{asythin}
T_p(\O_\eps)\approx\frac{\eps^{2p-1}}{2^p}\left(\frac{1}{p'+1}\right)^{p-1}\left(\int_A h^{p'+1}dx\right)^{p-1},\qquad
\lambda_p(\O_\eps)\approx\eps^{-p}\left(\frac{\pi_p}{\|h\|_\infty}\right)^p.
\ee
\end{prop}

\begin{proof}
First we deal with inequalities \eqref{ineqthin}. Let $\phi\in C^{\infty}_c(\O_\eps)$; since the function $\phi(x,\cdot)$ is admissible to compute $T_p\left(\eps h_-(x),\eps h_+(x)\right)$, by \eqref{def.tor} we obtain
\be\label{equa} 
\begin{split}
\int_{\eps h_-}^{\eps h_+} \phi(x,\cdot)dy&\le T^{1/p}_p(\eps h_-(x),\eps h_+(x))\left(\int^{\eps h_+(x)}_{\eps h_-(x)}|\nabla_y\phi(x,\cdot)|^p dy\right)^{1/p}\\
&\le T^{1/p}_p(\eps h_-(x),\eps h_+(x))\left(\int^{\eps h_+(x)}_{\eps h_-(x)}|\nabla\phi(x,\cdot)|^p dy\right)^{1/p}.
\end{split}
\ee
Taking into account \eqref{soltball} we have
\be\label{Tlin}
T_p\left(\eps h_-(x),\eps h_+(x)\right)=\frac{\eps^{2p-1}}{2^p}\left(\frac{p-1}{2p-1}\right)^{p-1}h^{2p-1}(x),
\ee
and thus, integrating on $A$ in \eqref{equa}, we deduce
$$\left(\int_{\O_\eps}\phi(x,y)\,dxdy\right)^p\le\frac{\eps^{2p-1}}{2^p}\left(\frac{p-1}{2p-1}\right)^{p-1}\bigg[\int_A h^{(2p-1)/p}\bigg(\int^{\eps h_+(x)}_{\eps h_-(x)}|\nabla\phi(x,\cdot)|^p\,dy\bigg)^{1/p}dx\bigg]^p.$$
H\"older inequality now gives
$$\left(\int_{\O_\eps}\phi(x,y)\,dxdy\right)^p\le\frac{\eps^{2p-1}}{2^p}\left(\frac{p-1}{2p-1}\right)^{p-1}\left(\int_A h^{(2p-1)/(p-1)}dx\right)^{p-1}\int_{\O_\eps}|\nabla\phi(x,y)|^pdxdy.$$
Since $\phi$ is arbitrary and $p'+1=(2p-1)/(p-1)$, we conclude that
$$T_p(\O_\eps)\le\frac{\eps^{2p-1}}{2^p}\left(\frac{1}{p'+1}\right)^{p-1}\left(\int_A h^{p'+1}dx\right)^{p-1}.$$
To get the second inequality in \eqref{ineqthin} we notice that, by \eqref{def.lambda}, for every $\phi\in C^\infty_c(\O_\eps)$ we have
$$\la_p(\eps h_-(x),\eps h_+(x))\int_{\eps h_-(x)}^{h_+(x)}|\phi(x,\cdot)|^{p} dy\le \int_{\eps h_-(x)}^{\eps h_+(x)}|\nabla_y \phi(x,\cdot)|^p dy \le \int_{\eps h_-(x)}^{\eps h_+(x)}|\nabla \phi(x,\cdot)|^p dy.$$
Since
$$\la_p(\eps h_-(x),\eps h_+(x))= h^{-p}(x)\eps^{-p}\pi_p^p\geq \|h\|_{\infty}^{-p}\eps^{-p}\pi_p^p,$$
integrating on $A$ and minimizing on $\phi$, we obtain
$$\la_p(\O_\eps)\ge \|h\|_{\infty}^{-p}\eps^{-p}\pi_p^p.$$
We now prove  \eqref{asythin} for $T_p(\O_\eps)$. 
To this end we consider the function
$$w_\eps (x,y):=\eps^{p'} h^{p'}(x)w\left(\frac{y-\eps h_-(x)}{\eps h(x)}\right),$$
where $w$ denotes the solution to \eqref{eq.pdetor} when $\O=(0,1)$ and $d=1$ (for the sake of brevity we omit the dependence on $p$). Notice that $w_\eps(x,\cdot)$ solves \eqref{eq.pdetor} in the interval $(\eps h_-(x),\eps h+(x))$. In particular, by using \eqref{def.tor2} and \eqref{Tlin}, we have
$$\int_{\eps h_-(x)}^{\eps h_+(x)}w_\eps(x,y)dy=\int_{\eps h_-(x)}^{\eps h_+(x)}|\nabla_y w_\eps(x,y)|^pdy=
\frac{\eps^{p'+1}h^{p'+1}(x)}{2^{p/(p-1)}(p'+1)}$$
A simple computation shows that
$$\nabla_y w_\eps(x,y)=\eps^{p'-1}W_1(x,y),\qquad\nabla_x w_\eps(x,y)=-\eps^{p'-1}W_1(x,y)\frac{y\nabla h(x)}{h(x)}+\eps^{p'} W_2(x,y),$$
where 
$$W_1(x,y)=h^{p'-1}(x)w'\left(\frac{y-\eps h_-(x)}{\eps h(x)}\right),$$
and
$$W_2(x,y)=p'h^{p'-1}(x)\nabla h(x) w\left(\frac{y-\eps h_-(x)}{\eps h(x)}\right)-h^{p'}(x)w'\left(\frac{y-\eps h_-(x)}{\eps h(x)}\right)\nabla\left(\frac{h_-(x)}{h(x)}\right).$$
In particular
\[\begin{split}
&\int_{\eps h_-(x)}^{\eps h_+(x)}|\nabla w_\eps(x,y)|^pdy=\\
&\int_{\eps h_-(x)}^{\eps h_+(x)}\left\{\left|\eps^{p'-1}W_1(x,y)\right|^{2}+\left|\eps^{p'-1}W_1(x,y)\frac{y\nabla h(x)}{h(x)}+\eps^{p'}W_2(x,y)\right|^2\right\}^{p/2}dy.
\end{split}\]
By exploiting the change of variable $z=\frac{y-\eps h_-(x)}{\eps h(x)}$ in the latter identity, we conclude that, as $\eps\to 0$, 
$$\int_{\eps h_-(x)}^{\eps h_+(x)}|\nabla w_\eps(x,y)|^p\,dy\approx\int_{\eps h_-(x)}^{\eps h_+(x)}|\nabla_y w_\eps(x,y)|^pdy.$$ 

Let $\phi\in C^\infty_c(A)$. Since  the function $v(x,y)=\phi(x)w_\eps(x,y)$ is admissible in \eqref{def.tor}, we get
\be\label{equiv1}
T(\O_\eps)\ge\left(\int_{\O_\eps}w_\eps(x,y)\phi(x)\,dxdy\right)^p\left(\int_{\O_\eps}\left|\nabla\left( w_\eps(x,y)\phi(x)\right)\right|^p\,dxdy\right)^{-1}.\ee
Moreover, by using basically the same argument as above, we have also that 
\be\label{equiv2}
\int_{\O_\eps} \left|\nabla\left(w_\eps(x,y)\phi(x)\right)\right|^p\,dxdy\approx \int_{\O_\eps} \left|\nabla_y w_\eps(x,y)\right|^p\left|\phi(x)\right|^pdxdy,\qquad\hbox{as }\eps\to0.
\ee
By combining \eqref{equiv1} and \eqref{equiv2} we obtain
$$\lim_{\eps\to 0} \frac{T_p(\O_\eps)}{\eps^{2p-1}}\ge 
\lim_{\eps\to 0}\frac{1}{\eps^{2p-1}} \left(\int_{\O_\eps}w_\eps(x,y)\phi(x)\,dxdy\right)^p\left(\int_{\O_\eps}\left|\nabla_y w_\eps(x,y)\right|^p\left|\phi(x)\right|^p\,dxdy\right)^{-1}.$$
Finally, by taking $\phi$ which  approximates $1_A$ in $L^p(A)$ in the right hand side of the inequality above, we conclude that   
$$\lim_{\eps\to0}\frac{T_p(\O_\eps)}{\eps^{2p-1}}\ge\frac{1}{2^p}\left(\frac{1}{p'+1}\right)^{p-1}\left(\int_A h^{p'+1}dx\right)^{p-1},$$
and the thesis is achived taking into account \eqref{ineqthin}. The asymtotics in \eqref{asythin} for $\lambda_p$ can be treated with similar  arguments.
\end{proof}

Actually, by means of a density argument, we can drop the regularity assumptions on $h_+$ and $h_-$ and extend the formulas \eqref{ineqthin} and \eqref{asythin} to any family $(\O_\eps)_{\eps>0}$ defined as in \eqref{thin}, with $h_+$ and $h_-$ bounded and measurable functions. We thus have:
$$F_{p,1}(\O_\eps)=\frac{\la_p(\O_\eps)T_p(\O_\eps)}{|\O_\eps|^{p-1}}\approx\gamma_p\left(\|h\|^{p'}_\infty\int_A h\,dx\right)^{1-p}\left(\int_A h^{1+p'}dx\right)^{p-1},$$
where
$$\gamma_p=\left(\frac {\pi_p }{2}\right)^p\left(\frac{1}{p'+1}\right)^{p-1}.$$ 
We then define the functional $F_{p,1}$ on the thin domain $(\O_\eps)_{\eps>0}$ associated with the $d-1$ dimensional domain $A$ and the local thickness function $h$ by
\be\label{defFthin}
F_{p,1}(A,h)=\gamma_p\left(\frac{\int_A h^{p'+1}dx}{\|h\|^{p'}_\infty\int_A h\,dx}\right)^{p-1}.
\ee

Our next goal is to give a complete solution to the optimization problems for the functional $F_{p,1}$ in the class of convex thin domains. To this aim we recall the following result (see Theorem 6.2 in \cite{bor73}).

\begin{theo}\label{theo.borell2}
Let $E\subset\R^N$ be a bounded open convex set, such that $0\in E$ and let $1\le s<r<\infty$. Then for every continuous function $h:E\to\R^+$ satisfying
\be\label{concave}
h(\la x)\ge\la h(x)+(1-\la)\qquad\forall x\in E,\ \forall\la\in (0,1),
\ee
and such that $\|h\|_{L^\infty(E)}=1$, it holds
$$\int_E h^r(x)\,dx\ge C_{r,s}\int_E h^s(x)\,dx$$
where
$$C_{r,s}=\frac{\int_0^1 (1-t)^{N-1}t^r\,dt}{\int_0^1 (1-t)^{N-1}t^s\,dt}.$$
In addition, equality occurs if $E$ is a ball of radius $1$ and $h(x)=1-|x|$.
\end{theo}

As an application we obtain the following lemma, which generalizes Proposition 5.2 in \cite{bbp20}.

\begin{lemm}\label{borell2}
Let $E\subset\R^N$ be a bounded open convex set and let $1<r<\infty$. Then for every concave function $h:E\to\R^+$ with $\|h\|_{L^\infty(E)}=1$ we have
\be\label{thingendim}
\frac{\begin{aligned}\int_E h^r(x)\,dx\end{aligned}}{\begin{aligned}\int_E h(x)\,dx\end{aligned}}\ge(N+1)\binom{N+r}{N}^{-1}\,.
\ee
In addition, the inequality above becomes an equality when $E$ is a ball of radius $1$ and $h(x)=1-|x|$.
\end{lemm}

\begin{proof}
First we assume that $E\subset\R^N$ is a ball centered in the origin and $h$ is a radially symmetric, decreasing, concave function $h:E\to[0,1]$ with $h(0)=1$. Then $h$ satisfies \eqref{concave} and we can apply Theorem \ref{theo.borell2} with $s=1$, to get
$$\int_E h^{r}(x)\,dx\ge C_{r,1}\int_E h(x)\,dx,$$
where
$$C_{r,1}=\frac{\int_0^1 (1-t)^{N-1}t^{r}dt }{\int_0^1 (1-t)^{N-1}t\,dt}=\frac{\binom{N+r}{N}^{-1}}{\binom{N+1}{N}^{-1}}=(N+1)\binom{N+r}{N}^{-1}.$$
In order to get the inequality \eqref{thingendim} in the general case, let $h^*:B\to[0,1]$ be the radially symmetric decreasing rearrangement of $h$, defined on the ball $B$ centered at the origin and with the same volume as $E$. The standard properties of the rearrangement imply that
$$\int_B (h^*)^{r}(x)\,dx=\int_E h^{r}(x)\,dx\,,\qquad\int_B h^*(x)\,dx=\int_E h(x)\,dx\,.$$
Moreover, it is well-known that $h^*$ is concave. Since $h^*$ satisfies all the assumptions of the previous case, we get that $h^*$ (hence $h$) satisfies ~\eqref{thingendim}. Finally, it is easy to show that the inequality in~\eqref{thingendim} holds as an equality for every cone function $h(x)=1-|x|$.
\end{proof}

We are now in a position to show the main theorem of this section.

\begin{theo}\label{theo.thinlow}
Let $1<p<\infty$. Then 
$$\begin{cases}
\sup\{F_{p,1}(A,h)\ :\ \HH^{d-1}(A)<+\infty,\ h\ge0\}=\gamma_p\\
\inf\{F_{p,1}(A,h)\ :\ A\hbox{ convex bounded, $h\ge0$, $h$ concave}\}= \gamma_pd^{p-1}\binom{d+p'}{d-1}^{1-p}.
\end{cases}$$
In addition, the first equality is attained taking $h(x)$ to be any constant function while the second equality is attained taking as $A$ the unit ball and as the local thickness function $h(x)$ the function $1-|x|$.
\end{theo}

\begin{proof}
Using definition \eqref{defFthin} it is straightforward to prove that
$$F_{p,1}(A,h)\le\gamma_p$$
and to verify that, if $h$ is constant, then
$$F_{p,1}(A,h)=\gamma_p.$$
Finally, by applying Lemma \ref{borell2} with $N=d-1$, $E=A$ and $r=p'+1$ we obtain the second part of the theorem.
\end{proof}

\section{The case $p=1$}\label{sp1}

Given an open set $\O\subset\R^d$ with finite measure we define its Cheeger constant $h(\O)$ as
\be\label{def.cheeger}
h(\O)=\inf\bigg\{\frac{P(E)}{|E|}\ :\ |E|>0,\ E\Subset\O\bigg\}.
\ee
where $E\Subset\O$ means that $\bar{E}\subset\O$. Notice that in definition \eqref{def.cheeger}, thanks to a well-known approximation argument, we can evaluate the quotient $P(E)/|E|$ among smooth sets which are compactly contained in $\O$. Following \cite{KaFr03} we have
\be\label{limp1}
\lim_{p\to1}\lambda_p(\O)=h(\O),
\ee
for every open set $\O$ with finite measure. 

\begin{rema}
A {\it caveat} is necessary at this point: the usual definition of Cheeger constant as
$$c(\O)=\inf\bigg\{\frac{P(E)}{|E|}\ :\ |E|>0,\ E\subset\O\bigg\}$$
is not appropriate to provide the limit equality \eqref{limp1}, which would hold only assuming a mild regularity on $\O$ (for instance, it is enough to consider $\O$ which coincides with its essential interior, see \cite{leo15}). To prove that in general $h(\O)\neq c(\O)$, one can consider $\O=B_1\setminus\partial B_{1/2}$. Then $c(\O)=c(B_1)=d$, while $h(\O)=2d$. The latter follows from the fact that, if $E\subset\O$, then $E=E_1\cup E_2$ where $E_1\Subset B_1\setminus \bar B_{1/2}$ and $E_2\Subset B_{1/2}$, together with the equality
$$
h(B_1\setminus \bar B_{1/2})=h(B_{1/2})=2d.$$
\end{rema}

By the same argument used in \cite{KaFr03} to prove \eqref{limp1} we can show that $T_p(\O)\to h^{-1}(\O)$ as $p\to1$. For the sake of completeness we give the short proof below.

\begin{prop}
Let $\O\subset\R^d$ be an open set with finite measure. Then, as $p\to 1$,
\be\label{eq.limp1} T_p(\O)\to h^{-1}(\O)\;.
\ee
\end{prop}

\begin{proof}
First we notice that for any $u\in C^{\infty}_c(\O)$, it holds:
\be\label{preCT}
\frac{\int_\O|\nabla u(x)|dx}{\int_\O |u(x)| dx}\ge h(\O).
\ee
Indeed, by assuming without loss of generality that $u\ge0$,  by coarea formula and Cavalieri's principle, we have that
$$\int_\O|\nabla u|\,dx=\int_0^{+\infty}\!\!\!\HH^{d-1}(\{u=t\})\,dt,\qquad\int_\O u\,dx=\int_0^{+\infty}\!\!\!|\{u>t\}|\,dt.$$
Since the sets $\{u>t\}\Subset\O$ are smooth for a.e. $t\in u(\O)$, \eqref{preCT} follows straightforwardly from \eqref{def.cheeger}. 
By combining \eqref{def.tor} with \eqref{preCT} and H\"older inequality we then have
\be\label{chtp}
|\O|^{1-p}T_p(\O)\le h^{-p}(\O),
\ee
for any $1<p<\infty$.

Now, let $E_k\Subset \O$ be a sequence of smooth sets of $\O$ such that $P(E_k)/|E_k|\to h(\O)$. For a fixed $k$ and any $\eps>0$ small enough, we can find a Lipschitz function $v$ compactly supported in $\O$, such that,
$$\chi_{E_k}\le v\le\chi_{E_{k,\eps}},\qquad|\nabla v|\le1/\eps\hbox{ in }E_{k,\eps}\setminus E_k,$$
where $E_{k,\eps}=E_k+B_\eps$. Hence, by \eqref{def.tor}, we have
$$T_p(\O)\ge\frac{\eps^{p}|E_k|^{p}}{|E_{k,\eps}\setminus E_k|}.$$
By first passing to the limit as $p\to 1$, and then as $\eps\to 0$ we get
$$\liminf_{p\to 1}T_p(\O)\ge \frac{|E_k|}{P(E_k)},$$
which implies, as $k\to\infty$, 
$$\liminf_{p\to 1}T_p(\O)\ge h^{-1}(\O).$$
Finally we conclude, taking into account \eqref{chtp}.
\end{proof}

The limits \eqref{limp1} and \eqref{eq.limp1} justify the following definition:
$$F_{1,q}(\O):=\left(h(\O)|\O|^{1/d}\right)^{1-q}.$$
Notice that $F_{p,q}(\O)\to F_{1,q}(\O)$ as $p\to 1$.
In the next proposition we solve the optimization problems for $F_{1,q}$ in both of the classes of general and convex domains.

\begin{prop}
For $0<q<1$, we have
$$\begin{cases}
\sup\big\{F_{1,q}(\O)\ :\ \O\subset\R^d\text{ open and convex },\ 0<|\O|<\infty\big\}=+\infty;\\
\min\big\{F_{1,q}(\O)\ :\ \O\subset\R^d\text{ open},\ 0<|\O|<\infty\big\}=F_{1,q}(B)
\end{cases}$$
For $q>1$, we have
$$\begin{cases}
\inf\big\{F_{1,q}(\O)\ :\ \O\subset\R^d\text{ open and convex},\ 0<|\O|<\infty\Big\}=0;\\
\max\big\{F_{1,q}(\O)\ :\ \O\subset\R^d\text{ open },\ 0<|\O|<\infty\big\}=F_{1,q}(B).
\end{cases}$$
\end{prop}

\begin{proof}
The minimality (respectively maximality) of $B$, for $0<q<1$ (respectively for $q>1$), is an immediate consequence of the well known inequality
$$h(B)|B|^{1/d}\le h(\O)|\O|^{1/d}.$$
which holds for any $\O\subset\R^d$ with finite measure.
To prove the other cases we use the inequality
$$h(\O)\ge \frac{P(\O)}{d |\O|},$$
which holds for any $\O\subset\R^d$ open, bounded, convex set (see \cite{bra18}, Corollary 5.2). Then taking $C_{A,\eps}$ as in \eqref{slab} we get
$$\lim_{\eps\to 0}h(C_{A,\eps})|C_{A,\eps}|^{1/d}=+\infty,$$
from which the thesis easily follows.
\end{proof}

\section{The case $p=\infty$}\label{spinfty}

The limit behaviour of the quantities $\la_p(\O)$, $T_p(\O)$, as $p\to \infty$, are well known for bounded open sets $\O\subset\R^d$: in \cite{FIN} and in \cite{JLM}  the authors prove that
\be\label{lapinfty}
(\lambda_p(\O))^{1/p}\to\frac{1}{\rho(\O)},
\ee
while,  following \cite{BBM} (see also  \cite{Ka90}) it holds $w_p\to d_{\O}$ uniformly in $\O$, which  implies
\be\label{payne}
(T_p(\O))^{1/p}\to\int_\O d_\O(x)\,dx.
\ee
Actually, in all these results, the boundedness assumption on $\O$ is not needed, as it is only used to provide the compactness of the embedding $W_0^{1,p}(\O)$ into the space $C_0(\O)$ defined as the completion of $C_c(\O)$ with respect to the uniform convergence. Indeed, this holds under the weaker assumption that $|\O|<+\infty$ (see Appendix \ref{sapp} for more details and for a $\Gamma$-convergence point of view of both limits \eqref{lapinfty} and \eqref{payne}).

According to \eqref{lapinfty} and to \eqref{payne}  we define the shape functional $F_{\infty,q}$ as 
\be\label{finftyq}
F_{\infty,q}(\O)=\frac{\big(\intbar_\O d_{\O}(x)\,dx\big)^q}{\rho(\O)|\O|^{(q-1)/d}}.
\ee

\begin{prop}\label{prop.Finf1}
 Let $\O\subset\R^d$ be an open convex set. Then
\be\label{d-bounds}
\frac{1}{d+1}\le\frac{1}{\rho(\O)}\intbar_\O d_\O(x)\,dx\le\frac12\;.
\ee
Moreover, both inequalities are sharp. In particular
$$\begin{cases}
\sup\big\{F_{\infty,1}(\O)\ :\ \O\text{ open and convex in }\R^d,\ 0<|\O|<\infty\big\}=1/2;\\
\min\big\{F_{\infty,1}(\O)\ :\ \O\text{ open and convex in }\R^d,\ 0<|\O|<\infty\Big\}=F_{\infty,1}(B)=1/(d+1).
\end{cases}$$
\end{prop}

For its proof, we recall the following result, for which we refer to \cite{bor73} and \cite{gar98}.

\begin{theo}\label{theo.borell}
Let $1\le q\le p$. Then for every convex set $E$ of $\R^N$ $(N\ge1)$ and every nonnegative concave function $f$ on $E$ we have
$$\Big[\intbar_E f^p\,dx\Big]^{1/p}\le C_{p,q}\Big[\intbar_E f^q\,dx\Big]^{1/q},$$
where the constant $C_{p,q}$ is given by
$$C_{p,q}=\binom{N+q}{N}^{1/q}\binom{N+p}{N}^{-1/p}.$$
In addition, the inequality above becomes an equality when $E$ is a ball of radius $1$ and $f(x)=1-|x|$.
\end{theo}

\begin{proof}[Proof of Proposition \ref{prop.Finf1}]
In order to prove the right-hand side inequality in \eqref{d-bounds}, for every $t\ge 0$, we denote by $\O(t)$ the \textit{interior parallel set} at distance $t$ from $\partial\O$, i.e.
$$\O(t):=\big\{x\in\O\ :\ d(x,\partial\O)>t\big\},$$
and by $A(t):=|\O(t)|$. Moreover we set
$$L(t):=P(\{x\in\O\ :\ d(x,\partial\O)=t\}).$$
Then for a.e. $t\in(0,\rho(\O))$ there exists the derivative $A'(t)$ and it coincides with $-L(t)$. Moreover, being $\O$ a convex set, $L$ is a monotone decreasing function. Then $A$ is a convex function such that $A(\rho)=0$ and $A(0)=|\O|$. As a consequence we have
$$A(t)\le A(0)\Big(1-\frac{t}{\rho}\Big)\qquad\text{on }[0,\rho].$$
Integrating by parts, we get
\[\begin{split}
\int_\O d_\O(x)\,dx&=\int_0^{\rho(\O)}tL(t)\,dt=-\int_0^{\rho(\O)}tA'(t)\,dt=\int_0^{\rho(\O)}A(t)\,dt\\
&\le\int_0^{\rho(\O)}A(0)\Big(1-\frac{t}{\rho}\Big)\,dt=\frac12\rho(\O)|\O|.
\end{split}\]
The value $1/2$  is asymptotically attained in \eqref{d-bounds} by considering a sequence of slab domains
$$\O_\eps:=(0,1)^{d-1}\times (0,\eps)\subset\R^d,$$
as $\eps\to0$. Indeed, we have $\rho(\O_{\eps})=\eps/2$ and $|\O_{\eps}|=\eps$. Being
$$A_{\eps}(t)=|(t,1-t)^{d-1}\times (t,\eps-t)|= (1-2t)^{d-1}(\eps-2t)$$
we get
$$\lim_{\eps\to0}\frac{\intbar_{\O_\eps}d_{\O_\eps}(x)\,dx}{\rho(\O_\eps)}=\lim_{\eps\to0}\frac{\int_0^{\eps/2}A_\eps(t)\,dt}{\eps^2/2}=\lim_{\eps\to0}\bigg[\frac{(1-\eps)^{d+1}+\eps (d+1)-1}{(d+1)d\eps^2}\bigg]=1/2.$$

Now we prove the left-hand side inequality in \eqref{d-bounds}. Since $\O$ is convex, the distance function $d_\O$ is concave (see \cite{AK}); then, applying Theorem \eqref{theo.borell} to $d_\O$, we obtain
\be\label{p-norma}
\Big[\intbar_\O d_\O^p\,dx\Big]^{1/p}\le C_{p,1}\intbar_\O d_\O\,dx\quad\forall p\ge1.
\ee
Since \eqref{p-norma} is an identity when $\O=B$, $C_{p,1}$ satisfies
$$C_{p,1}=\frac{\|f\|_p}{\|f\|_1}\omega_d^{1-1/p}\to\frac{\omega_d}{\|f\|_1}=(d+1).$$
As $p\to\infty$ in \eqref{p-norma}, we obtain
$$\rho(\O)\le(d+1)\intbar_\O d_\O\,dx$$
which is an equality when $\O=B$.
\end{proof}

\begin{rema}
The proof of the right-hand side of \eqref{d-bounds} relies on the convexity properties of the function $A(t)$. In the planar case a general result, due to Sz. Nagy (see \cite{Nagy59}), ensures that, if $\O$ is any bounded $k$\textit{-connected} open set, (i.e. $\O^c$ has $k$ bounded connected components), then the function
$$t\mapsto A(t)+2\pi(k-1)t^2,\qquad t\in(0,\rho(\O))$$ 
is convex. Therefore, for such an $\O$, with the same argument as above it is easy to prove that
$$F_{\infty,1}(\O)\le 1/2+(k-1)\pi/6.$$
Hence, it is interesting to notice how, even when $k=0,1$, the upper bound given in \eqref{d-bounds} remains sharp. In other words, in the maximization of $F_{\infty,1}$ on planar domains, there is no gain in replacing the class of convex domains by the larger one consisting of simply-connected domains or even more in allowing $\O$ to have a single hole.
\end{rema}

In the general case $q\ne1$ the optimization problems for the functional $F_{\infty,q}$ defined in \eqref{finftyq} are studied below.

\begin{coro}\label{prop.Finfq}
If $0<q<1$, then 
$$\begin{cases}
\sup\big\{F_{\infty,q}(\O)\ :\ \O\text{ open and convex in }\R^d,\ 0<|\O|<\infty\big\}=\infty;\\
\min\big\{F_{\infty,q}(\O)\ :\ \O\text{ open and convex in }\R^d,\ 0<|\O|<\infty\Big\}=F_{\infty,q}(B)=(d+1)^{-q}\omega_d^{(1-q)/d}.
\end{cases}$$
If $q>1$, then 
$$\begin{cases}
\sup\big\{F_{\infty,q}(\O)\ :\ \O\text{ open and convex in }\R^d,\ 0<|\O|<\infty\big\}\le(1/2)^q\omega_d ^{(1-q)/d} ;\\
\inf\big\{F_{\infty,q}(\O)\ :\ \O\text{ open and convex in }\R^d,\ 0<|\O|<\infty\Big\}=0\;.
\end{cases}$$
\end{coro}

\begin{proof}
Notice that 
\be\label{eq.decoFinfq} 
F_{\infty, q}(\O)=F_{\infty,1}^{q}(\O)\left(\frac{\rho(\O)}{|\O|^{1/d}}\right)^{q-1},
\ee
and that the inequality $|\O|\ge\omega_d\rho(\O)^d$ holds for every open set $\O\subset\R^d$ with equality when $\O=B$. Thus, if $0<q<1$, by \eqref{d-bounds} we have
$$F_{\infty,q}(\O)\ge F_{\infty,q}(B)=(d+1)^{-q}\omega_d^{(1-q)/d},$$
while if $q>1$, using again \eqref{d-bounds} we have
$$F_{\infty,q}(\O)\le(1/2)^q\omega_d^{(1-q)/d}.$$
Finally, let $\O_\eps$ be the slab domain as in Proposition \ref{prop.Finf1}. Then 
$$\lim_{\eps\to 0}F_{\infty,q}(\O_\eps)=(1/2)^{-q}\lim_{\eps\to 0}\left(\frac{\eps}{2\eps^{1/d}}\right)^{q-1}=
\begin{cases}
0,&\hbox{if }q>1;\\
\infty,&\hbox{if }0<q<1,
\end{cases}$$
from which the thesis is achieved.
\end{proof}

If we remove the convexity assumption on the admissible domains $\O$ it is easy to show that the minimization problem for $F_{\infty,q}$ is always ill posed. Indeed, if $q>1$ this follows directly by Corollary \eqref{prop.Finfq} while, if $q<1$, taking into account \eqref{eq.decoFinfq} and by consider $\O_n$ to be the union of $n$ disjoint balls of fixed radii, we get $F_{\infty,q}(\O_n)\to0$ as $n\to\infty$. Similarly, if $q=1$, taking $\O_n$ to be the union of $n$ disjoint balls of radius $r_j=j^{-\alpha}$, where $1/(d+1)<\alpha<1/d$, we have $F_{\infty,1}(\O_n)\to0$, as $n\to\infty$. Concerning the upper bound, we trivially have
$$F_{\infty,1}(\O)\le1,$$ 
and, as a consequence, (using again \eqref{eq.decoFinfq}), when $q\ge 1$, we have
$$F_{\infty,q}(\O)\le\omega_d^{(1-q)/d}.$$
However, working with general domains provides an upper bound larger than in \eqref{d-bounds}; for instance, in the two-dimensional case, taking as $\O_N$ the unit disk where we remove $N$ points as in Figure \ref{fig1}, gives
$$\lim_{N\to\infty}F_{\infty,1}(\O_N)=\intbar_E|x|\,dx=\frac13+\frac{\log3}{4}\approx0.608$$
where $E$ is the regular exagon with unitary sides centered at the origin, as an easy calculation shows.

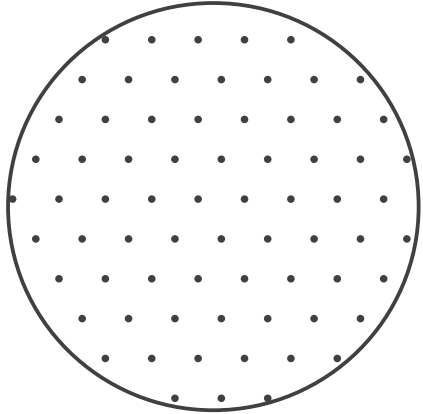
\begin{figure}[h!]
\centering
\begin{tikzpicture}
\coordinate (C) at (0.2, -0.1); 
\begin{scope}
\clip (C) circle (\raggiocerchio);
\foreach \x in {-\MM, ..., \MM} {
\foreach \y in {-\MM, ..., \MM} {
\fill[black!75, x=(0:\distanzapunti), y=(60:\distanzapunti)] (\x, \y) circle (0.5mm);}}
\end{scope}
\draw[black!75, line width=0.5mm] (C) circle (\raggiocerchio);
\end{tikzpicture}

\caption{The two-dimensional region $\O_N$.}\label{fig1}
\end{figure}

\section{Further remarks and open questions}\label{further}

Several interesting problems and questions about the shape functionals $F_{p,q}$ are still open; in this section we list some of them.

\bigskip

{\bf Problem 1. }The characterization of the infimum of $F_{p,q}$ in the class of all domains is well clarified in Proposition \ref{prop.infall}; on the contrary, for the supremum of $F_{p,q}$, Proposition \ref{prop.supall} only says it is finite for $q\ge1$. It would be interesting to know if the supremum can be better characterized, if it is attained, and in particular if it is attained for a ball when the exponent $q$ is large enough (see also Problems 1 and 2 in \cite{bbp20}).

\bigskip

{\bf Problem 2. }Concerning Problem 1 above, the case $q=1$ is particularly interesting. The Polya inequality gives $\sup F_{p,1}\le1$, and Proposition \ref{sharppolya} gives $\sup F_{p,1}=1$ whenever $p\le d$. It would be interesting to prove (or disprove) that $\sup F_{p,1}<1$ for all $p>d$.

\bigskip

{\bf Problem 3. }In the convex setting, Proposition \ref{prop.p1convex} provides some upper and lower bounds to $F_{p,1}$ that however are far from being sharp. Even in the case $p=2$, sharp values for the infimum and the supremum of $F_{2,1}$ in the class of convex sets are unknown (see Conjecture 4.2. in \cite{bbp20}). It seems natural to conjecture that the right sharp inequalities are those given in Theorem \ref{theo.thinlow} for $F_{p,1}$ on the class of thin domain.

\bigskip

{\bf Problem 4. }In the two-dimensional case with $p=\infty$ we have seen that the domains $\O_N$ in Figure \ref{fig1} give for the shape functional $F_{\infty,1}$ the asymptotic value $\frac13+\frac{\log3}{4}$. It would be interesting to prove (or disprove) that this number is actually the supremum of $F_{\infty,1}(\O)$ when $\O$ varies in the class of all bounded open two-dimensional sets. In addition, in the case of a dimension $d>2$, it is not clear how a maximizing sequence $(\O_n)$ for $F_{\infty,1}$ has to be.

\appendix\section{}\label{sapp}
We devote this Appendix to give a proof of  the known asymptotics  \eqref{lapinfty} and \eqref{payne}  by means of $\Gamma$-convergence  when $0<|\O|<+\infty$. 
We  recall that if $p>d$ and $\O\subset\R^d$ is any (possibly unbounded) open set with finite measure we have the compact embedding:
\be\label{compact}
W^{1,p}_0(\O)\hookrightarrow C_0(\O).
\ee
A quick proof of it can be obtained by combining the Gagliardo-Niremberg inequality in $W^{1,p}(\R^d)$:
\be\label{GN}
\Vert\phi\Vert_{L^{\infty}(\R^d)}\le C(d,p)\Vert\nabla\phi\Vert^{1-\alpha}_{L^{p}(\R^d)}\Vert\phi\Vert^{\alpha}_{L^{p}(\R^d)}\,\qquad\forall\phi\in W^{1,p}(\R^d),
\ee
together with the well known facts that the inclusion $W^{1,p}_0(\O)\hookrightarrow L^p(\O)$ is compact (thanks to the Riesz-Frech\'et-Kolmogorov Theorem) and the embedding \eqref{compact} is continuous. Note that, in \eqref{GN}, $C(d,p)$ denotes a  positive constant depending on $p$ and $d$, and $\alpha=1-d/p$ (we refer to \cite{Ad}, Chapter $6$, for a comprehensive discussion on necessary and sufficient conditions for the compactness of several embeddings  of Sobolev spaces).

In particular, if we denote by $W_0^{1,\infty}(\O)$ the closure of $C_c^\infty(\O)$ with respect to the weak* convergence of $W^{1,\infty}(\O)$, we have that  $u\in W^{1,\infty}_0(\O)$  if and only if $u\in C_0(\O)$ and $u$ is a Lipschitz continuous function on $\O$. Moreover $W_0^{1,\infty}(\O)$ can be easily characterized as:
\be\label{cap}
W^{1,\infty}_0(\O)=W^{1,\infty}(\O)\cap\bigcap_{p\ge 1}W^{1,p}_0(\O).
\ee

\begin{prop}\label{Gamma1}
Let $\O\subset\R^d$ be an open set with finite measure and let $\Psi_p,\Psi_\infty: L^{1}(\O)\to\bar\R$ be defined by
\[\begin{split}
&\Psi_p(u):=\begin{cases}
\|\nabla u\|_{L^{p}(\O)}&\hbox{if }u\in W_0^{1,p}(\O)\ \text{and }\|u\|_p=1,\\
+\infty&\hbox{otherwise},
\end{cases}\\
&\Psi_\infty(u):=\begin{cases}
\|\nabla u\|_{L^{\infty}(\O)}&\hbox{if }u\in W^{1,\infty}_0(\O),\ \| u\|_\infty=1,\\
+\infty&\hbox{otherwise}.
\end{cases}
\end{split}\]
Then, as $p\to\infty$, the sequence $\Psi_p$ $\Gamma$-converges to $\Psi_\infty$ with respect to the $L^1$-convergence.
\end{prop}

\begin{proof}
Let $p_n\to\infty$. The $\Gamma$-$\limsup$ inequality is trivial since, for every $u\in W^{1,\infty}_0(\O)$ with $\|u\|_\infty=1$, the sequence $u_{p_n}=\|u\|_{p_n}^{-1}u$ converges  to $u$ in $L^1$ and satisfies
$$\|u_n\|_{p_n}=1,\quad\limsup_{n\to\infty}\Psi_{p_n} (u_n)= \Psi_\infty(u).$$
To prove the $\Gamma$-$\liminf$ inequality, without loss of generality, let $u\in L^{\infty}(\O)$,  $(u_{p_n})\subseteq  W^{1,p_n}_0(\O)$ be such that $u_{p_n}\to u$ in $L^1(\O)$, $\|u_{p_n}\|_{p_n}=1$,  and  $\liminf_{n\to\infty}\Psi_{p_n}(u_{p_n})=C<\infty$.  Since for every $q\ge1$ and for $n$ large enough it holds
$$\|Du_{p_n}\|_{q}\leq |\O|^{ 1/q-1/{p_n}}\Psi_{p_n}(u_{p_n})$$
we get that $u_{p_n}\to u$ in $L^q(\O)$, $u\in W^{1,q}_0(\O)$ and
\be\label{normeq}
\|Du\|_{q}\leq C |\O|^{ 1/q}.
\ee
Moreover  
$$|\O|^{1/q}=\lim_{n\to\infty}|\O|^{1/q-1/p_n}\|u_{p_n}\|_p\ge \lim_{n\to \infty}\|u_{p_n}\|_q=\|u\|_q,$$
which yields, as $q\to\infty$,  $\|u\|_{\infty}\le 1$. Combining this estimate with \eqref{normeq}, we get that  $u\in W^{1,\infty}(\O)$; hence,  by \eqref{cap}, $u\in W^{1,\infty}_0(\O)$. Thanks to the compact embedding of  $W_0^{1,q}(\O)$ in $C_0(\O)$ when $q>d$, we obtain that $\|u_{p_n}- u\|_{\infty}\to 0$ as $n\to \infty$ and, since
$$1=\|u_n\|_{p_n}\le |\O|^{1/p_n}\|u_{p_n}\|_\infty\,,$$
we get that $\|u\|_\infty=1$. Finally, by letting $n\to \infty$ in \eqref{normeq}, it follows
$$\Psi_\infty (u)=\|Du\|_\infty\leq C=\liminf_{n\to\infty}\Psi(u_{p_n})  .$$
\end{proof}

\begin{coro}\label{coro.la}
Let $\O\subset\R^d$ be an open set with finite measure. Then, as $p\to \infty$,
$$\la^{1/p}_p(\O)\to \rho(\O)^{-1}.$$
\end{coro}

\begin{proof}
Using $d_{\O}$ as a test function for \eqref{def.lambda} we have $\limsup_{p\to\infty}\la_p(\O)\le \rho(\O)^{-1}$.
Moreover we notice that, for any $\phi\in W^{1,\infty}_0(\O)$, it holds
\be\label{ineq}
|\phi(x)|\le d_{\O}(x)\|\nabla \phi\|_{\infty}.
\ee
Let $u_p$ be the (only) minimum of $\Psi_p$, and let $p_n\to \infty$. With the same argument of Proposition \ref{Gamma1} we can  assume $u_{p_n}\to u_\infty$ uniformly in $L^{1}(\O)$, where $u_\infty\in W^{1,\infty}_0(\O)$ and $\|u_{\infty}\|_{\infty}=1$. Then, by Proposition \ref{Gamma1} and by \eqref{ineq},  we have
$$\rho(\O)^{-1}\le \|\nabla u_{\infty}\|=\Psi(u_\infty)\le \liminf_{n\to\infty}\Psi(u_{p_n})=\liminf_{p\to\infty}\la_{p_n}^{1/p_n}(\O).$$
The thesis follows  by the arbitrariness  of the sequence $p_n$.
\end{proof}

Next Proposition generalizes Proposition 2.1 in \cite{GNP}.

\begin{prop}\label{Gamma2} Let $\O\subset\R^d$ be an open set with finite measure and let $\Phi_p,\Phi_\infty:L^1(\O)\to\bar\R$ be defined by
\[\begin{split}
&\Phi_p(u):=\begin{cases}
\frac1p\int_\O|\nabla u|^p\,dx&\hbox{if }u\in W_0^{1,p}(\O),\\
+\infty&\hbox{otherwise},
\end{cases}\\
&\Phi_\infty(u):=\begin{cases}
0&\hbox{if }u\in W^{1,\infty}_0(\O),\ \|\nabla u\|_\infty\le1,\\
+\infty&\hbox{otherwise}.
\end{cases}
\end{split}\]
Then, as $p\to\infty$, the functionals $\Phi_p$ $\Gamma$-converge to $\Phi_\infty$ with respect to the $L^1$-convergence. Moreover, we have
$$\lim_{p\to\infty}\Phi_p(u)=\Phi_\infty(u)\qquad\hbox{for every }u\in L^1(\O).$$
\end{prop}

\begin{proof}
Let $p_n\to\infty$. The $\Gamma$-limsup inequality is trivial since for any $u\in W^{1,\infty}_0(\O)$ with $\|\nabla u\|_\infty\le1$ we have $u\in W_0^{1,p_n}(\O)$ for every $n\in\N$ and thus
$$\limsup_{n\to\infty}\Phi_{p_n}(u)=\limsup_{n\to\infty}\int_\O\frac{1}{p_n}|\nabla u(x)|^{p_n}\,dx\le|\O|\limsup_{n\to\infty}\frac{1}{p_n}=0=\Phi_\infty(u).$$
To prove the $\Gamma$-liminf inequality, we can assume $u_n,u\in L^1(\O)$, $u_n\to u$ in $L^1(\O)$, $u_n\in W_0^{1,p_n}(\O)$, and
$$\liminf_{n\to\infty}\Phi_{p_n}(u_{p_n})=\lim_{n\to\infty}\Phi_{p_n}(u_{p_n})=M<+\infty.$$
If $q>1$ and $p_n>q$ we have
$$\|\nabla u_{p_n}\|_q\le|\O|^{1/q-1/{p_n}}\|\nabla u_{p_n}\|_{p_n}\le|\O|^{1/q-1/{p_n}}(M p_n)^{1/p_n}$$
which forces $u\in W^{1,q}_0(\O)$.
Moreover, since
$$\|\nabla u\|_q\le\liminf_{n\to\infty}\|\nabla u_{p_n}\|_q\le|\O|^{1/q},$$
we have also $\|\nabla u\|_\infty\le1$. Therefore, by \eqref{cap}, $u\in W_0^{1,\infty}(\O)$ and 
$$\Phi_\infty(u)=0\le\liminf_{n\to\infty}\Phi_{p_n}(u_{p_n}).$$
The thesis follows  by the arbitrariness  of the sequence $p_n$.
\end{proof}

\begin{coro}
Let $\O\subset\R^d$ be an open set with finite measure, let $w_p$ be the solution to \eqref{eq.pdetor}. Then, as $p\to\infty$, 
$$w_p\to d_\O\ \text{in } L^\infty(\O), \qquad  (T_p(\O))^{1/p}\to\int_\O d_\O(x)\,dx.$$ 
\end{coro}

\begin{proof}
It is sufficient to show that $w_p\to d_\O$ uniformly in $\O$.  First we notice that
by \eqref{prop.polya} we get
\be\label{eq.pnorme}
\Big(|\O|^{-1}\int_\O |\nabla w_{p}|^{p}\,dx\Big)^{1/p}=\big(|\O|^{1-p}T_{p_n}(\O)\big)^{1/(p(p-1))}\le\Big(\lambda^{1/p}_{p}(\O)\Big)^{1/1-p}.
\ee
By Corollary \ref{coro.la} we have
$$C:=\sup_p|\O|^{-1/{p}}\Vert\nabla w_{p}\|_{p}<+\infty.$$    
Moreover, for every fixed $q\geq 1$ and $p$ large enough, by H\"older inequality, we have that 
\be\label{eq.rnorme}
\|\nabla w_{p}\|_q\le \|\nabla w_{p}\|_{p}|\O|^{1/q-1/{p}}\le C|\O|^{1/q}.
\ee
Let $p_n\to \infty$. By applying \eqref{eq.rnorme} we can show that there exists $w_{\infty}\in W^{1,\infty}_0(\O)$, such that $w_{p_n}$ converges uniformly to  $w_\infty$ and weakly in  $W^{1,q}(\O)$ for every $q\geq 1$.
Notice that \eqref{eq.pnorme} combined with Corollary \ref{coro.la} shows also
$$\|\nabla w_\infty\|_q\le|\O|^{1/q},$$ for every $q\geq 1$, which implies $\|w_\infty\|\le1$.

Now let $J_p$ be the functional defined in \eqref{def.Jp} and $J_\infty$ be the functional given by 
$$J_\infty(u):=\Phi_{\infty}(u)-\int_\O u\,dx.$$
Since the functional $u\mapsto\int_\O u\,dx$ is continuous with respect to the $L^1$-convergence, thanks to Proposition \ref{Gamma2}, we have that
$$\lim_{n\to\infty}J_{p_n}(w_{p_n})=J_\infty(w_\infty)=\min_{u\in W_0^{1,\infty}(\O)\,, \|\nabla u\|_\infty\le1}J_\infty(u)=-\int_\O w_\infty\,dx.$$
Moreover, by using \eqref{ineq} we have $w_\infty(x)\le d_\O(x)$. In addition, since $d_\O\in W^{1,\infty}_0(\O)$, we have also $J_\infty(w_\infty)\le J_\infty(d_\O)$, i.e. $\int_\O(w_\infty-d_\O)\,dx\ge0$. Hence $d_\O=w_\infty$. By the arbitrariness of the  sequence $p_n$, we get that $w_p\to d_{\O}$ uniformly as $p\to\infty$.
\end{proof}

\bigskip

\noindent{\bf Acknowledgments.} The work of GB is part of the project 2017TEXA3H {\it``Gradient flows, Optimal Transport and Metric Measure Structures''} funded by the Italian Ministry of Research and University. The authors are member of the Gruppo Nazionale per l'Analisi Matematica, la Probabilit\`a e le loro Applicazioni (GNAMPA) of the Istituto Nazionale di Alta Matematica (INdAM).

\bigskip
{\small\noindent
Luca Briani:
Dipartimento di Matematica,
Universit\`a di Pisa\\
Largo B. Pontecorvo 5,
56127 Pisa - ITALY\\
{\tt luca.briani@phd.unipi.it}

\bigskip\noindent
Giuseppe Buttazzo:
Dipartimento di Matematica,
Universit\`a di Pisa\\
Largo B. Pontecorvo 5,
56127 Pisa - ITALY\\
{\tt giuseppe.buttazzo@dm.unipi.it}\\
{\tt http://www.dm.unipi.it/pages/buttazzo/}

\bigskip\noindent
Francesca Prinari:
Dipartimento di Scienze Agrarie, Alimentari e Agro-ambientali,
Universit\`a di Pisa\\
Via del Borghetto 80,
 56124 Pisa - ITALY\\
{\tt francesca.prinari@unipi.it}

\end{document}